\documentclass[11pt, a4paper, twoside,reqno]{amsart}

\usepackage{tikz-cd}
\usetikzlibrary{matrix,decorations.pathreplacing, calc, positioning,fit}
\usepackage{amsmath}
\usepackage{amsfonts}
\usepackage[shortlabels]{enumitem}
\usepackage{amssymb}
\usepackage{amsthm}
\usepackage{graphicx}
\usepackage[utf8]{inputenc}
\usepackage{textcomp}
\usepackage[squaren,Gray]{SIunits}
\usepackage{float,amsmath,amssymb,array,import,graphicx,dsfont}
\usepackage{hyperref}
\usepackage[title,toc,page,header]{appendix} 
\usepackage{mathtools} 
\usepackage{lmodern}  
\usepackage[T1]{fontenc} 
\usepackage[a4paper]{geometry} 
\usepackage{graphicx} 
\usepackage{pgf,tikz}
\usepackage{stackrel}
\hypersetup{backref=true, pagebackref=true, colorlinks=true, linkcolor=black, urlcolor=black} 
\usetikzlibrary{arrows}
\usepackage{fancyhdr}
\usepackage{faktor}
\usepackage{lettrine} 
\usepackage[all]{xy}
\usepackage{mathrsfs}
\usepackage{enumitem}

\newcommand{\p}{\mathfrak{p}}

\newcommand{\Q}{\mathbb{Q}}
\newcommand{\N}{\mathbb{N}}

\newcommand{\Z}{\mathbb{Z}}

\newtheorem{defi}{Definition}
[section]
\newtheorem{thm}[defi]{Theorem}
\newtheorem{lem}[defi]{Lemma}
\newtheorem{coro}[defi]{Corollary}
\newtheorem{prop}[defi]{Proposition}
\newtheorem{conj}[defi]{Conjecture}
\theoremstyle{definition}

\newtheorem{remrk}[defi]{Remark}

\newtheorem*{lem*}{Lemme}
\newtheorem*{prop*}{Proposition}
\newtheorem*{thm*}{Theorem}
\newtheorem*{remrk*}{Remark}

\newtheorem*{nota}{Notation}

\numberwithin{exos}{subsection}

\usepackage[backend=biber, style=alphabetic, sorting=ynt]{biblatex}

\usepackage{calrsfs}

\title{A note on P\'olya groups}

\date{}

\usepackage[backend=biber, style=alphabetic, sorting=ynt]{biblatex}
\begin{document}
\title{A note on P\'olya groups}
    
    \author[Étienne Emmelin]{Étienne Emmelin}
    \address{Normandie Univ, UNICAEN, CNRS, LMNO, 14000 Caen, France}
    \email{etienne.emmelin@unicaen.fr}
        
    \date{\today}

    

\begin{abstract}
Let $K/\Q$ be a finite Galois extension. The P\'olya group of $K$ is the subgroup of the class group $Cl(K)$, generated by the classes of ambiguous ideals of $K$. In this note, among other results, we prove that every finite abelian group is isomorphic to the P\'olya group of a number field.

\end{abstract}
\maketitle

\section{Introduction}
\begin{defi}[\cite{zan}]
Let $K$ a number field and $l$ a positive integer. Define
\begin{equation*}
    I_K^l:=\underset{\mathfrak{N}(\p)=l}{\prod_{\p\; \text{prime of $K$}}}\p,
\end{equation*}
the ideal of $\mathcal{O}_K$ formed by the product of all primes of $K$ with norm $l$. When $l$ is not a norm of prime ideal, then $I_K^l=\mathcal{O}_K$.
A number field is called a P\'olya field when all the ideals $I_K^l$ are principal.
\end{defi}
\begin{defi}[\cite{cah}, II.4]
    The P\'olya group of a number field is a subgroup of the class group of $K$ generated by the classes of the ideals of $I_K^l$. This subgroup is noted $Po(K)$.
\end{defi}

In our study, we will only consider number fields Galois over $\Q$. This hypothesis gives us an other definition of P\'olya group in term of ambiguous ideals.\\

Let $K/\Q$ a finite Galois extension and $G$ its Galois group. Then, the group of ambiguous ideals $\mathcal{I}(K)^G$ of $K$, i.e. of ideals invariant under the action of $G$, is generated by the primes of the form   
\begin{equation*}
    \sqrt{p\mathcal{O}_K}:=\underset{\p \;\text{prime of K}}{\prod_{\p|p}}\p=(I_K^{p^{f_p(K/\Q)}}),
\end{equation*}
where $p$ is a positive prime number and $f_p(K/\Q)$ the inertia degree of each prime $\p |p$.\\
Consequently, the P\'olya group of $K$ is generated by the classes of the ambiguous ideals. For simplicity, for every prime number $p$, we will use the notation $\sqrt{p\mathcal{O}_K}$ instead of $I_K^{p^{f_p(K/\Q)}}$.\\

 We give now some results about the behavior of the P\'olya groups in linearly disjoint Galois extensions to obtain the main proposition of this section (see for instance \cite{al} or \cite{jlc}). 
\begin{nota}
    Let $L/K$ a finite extension of number fields. Consider the injective morphism:
    \begin{equation*}
    \begin{array}{ccc}
      j_K^L: \mathcal{I}(K) &\to & \mathcal{I}(L) \\
       I &\mapsto & I\mathcal{O}_L .
    \end{array}
    \end{equation*}
It induces a morphism:
\begin{equation*}
    \begin{array}{ccc}
         \varepsilon_K^L: Cl(K)&\to  &Cl(L)\\
         \overline{I}&\mapsto &\overline{I\mathcal{O}_L}.
    \end{array}
\end{equation*}
\end{nota}
If $L$ and $K$ are  Galois extensions over $\Q$, the P\'olya group of $K$ and $L$ behave well with respect to this morphism.
\begin{prop}
    Let $K\subseteq L$ are two Galois extensions over $\Q$, then
    \begin{equation*}
        \varepsilon_K^L(Po(K))\subseteq Po(L).
    \end{equation*}
\end{prop}

For the next two part, we will need some results, directly deduced from Abhyankar's Lemma. In \cite{sti} 3.9.1, the Theorem is proved for function fields over finite fields but the proof remains valid for number fields.
\begin{lem}[Abhyankar] Let $K_1$ and $K_2$ two number fields, $L=K_1K_2$ and $K$ a subfield of $K_1\cap K_2$. Let $\mathfrak{q}$ a prime of $L$ such that $\mathfrak{q}\cap K=\mathfrak{p}$. Let $p\Z=\mathfrak{q}\cap \Z$, $\mathfrak{p}_1=\mathfrak{q}\cap K_1$ and $\mathfrak{p}_2=\mathfrak{q}\cap K_2$. If $p$ does not divide one of the ramification indices $e(\mathfrak{p}_1/\mathfrak{p})$ or  $e(\mathfrak{p}_2/\mathfrak{p})$, then $e(\mathfrak{q}/\p)=\text{lcm}(e(\mathfrak{p}_1/\mathfrak{p}),e(\mathfrak{p}_2/\mathfrak{p}))$.
    
\end{lem}
\begin{lem}\label{l1.5}
    Let $K/\Q$ a Galois extension, then $|Po(K)^{[K:\Q]}|=\{0\}$.
\end{lem}
\begin{proof}
    Let $p$ a prime number and $e_p(K/\Q)$ the ramification index of each prime $\mathfrak{P}$ of $K$ lying over $p$. As $e_p(K/\Q)|[K:\Q]$ we obtain that $\sqrt{p\mathcal{O}_K}^{[K:\Q]}$ is principal.
\end{proof}
For the convenience of the reader we give a detailed proof of the next two results.
\begin{coro}\label{1.4}
 Let $K_1$ and $K_2$ two  finite Galois extensions of $\Q$, and $L=K_1K_2$. If for each prime number p, at least one extension
$Ki/\Q$ ($i=1,2$) is tamely ramified with respect to $p$, then:
 \begin{equation*}
        Po(L)=\varepsilon_{K_1}^L(Po(K_1))\cdot \varepsilon_{K_2}^L(Po(K_2)).
            \end{equation*}
\end{coro}

\begin{proof}(See also \cite{jlc} Lemma 3.2). 
    Let $p$ a prime number that ramifies in $L$. Let $m=e_p(L/\Q)$, $e_1=e_p(K_1/\Q)$ and $e_2=e_p(K_2/\Q)$, the ramification indices of primes lying over $p$ in the different extensions. By Abhyankar's Lemma, $m=\text{lcm}(e_1,e_2).$ Then we have:
    \begin{equation*}
        \sqrt{p\mathcal{O}_{K_1}}\mathcal{O}_L=\sqrt{p\mathcal{O}_L}^{\frac{m}{e_1}}\;\text{and}\;\sqrt{p\mathcal{O}_{K_2}}\mathcal{O}_L=\sqrt{p\mathcal{O}_L}^{\frac{m}{e_2}}.
    \end{equation*}
But $\frac{m}{e_1}$ and $\frac{m}{e_2}$ are coprime, thus there exist $u,v\in\Z$ such that $u\frac{m}{e_1}+v\frac{m}{e_2}=1$. Therefore
\begin{equation*}
    \sqrt{p\mathcal{O}_L}=\sqrt{p\mathcal{O}_L}^{u\frac{m}{e_1}}\sqrt{p\mathcal{O}_L}^{v\frac{m}{e_2}}=(\sqrt{p\mathcal{O}_{K_1}}\mathcal{O}_L)^u(\sqrt{p\mathcal{O}_{K_2}}\mathcal{O}_L)^v.
\end{equation*}
We then deduce 
\begin{equation*}
    \langle \sqrt{p\mathcal{O}_{K_1}}\mathcal{O}_L,\sqrt{p\mathcal{O}_{K_2}}\mathcal{O}_L\rangle=\langle \sqrt{p\mathcal{O}_L}\rangle\; .\end{equation*}

\end{proof}

\begin{coro}\label{1.7}
     Let $K_1$ and $K_2$ two  finite Galois extensions of $\Q$ such that the degrees $[K_1:\Q]$ and $[K_2:\Q]$ are coprime, and let $L=K_1K_2$. Then:
     \begin{equation*}
         Po(L)\simeq Po(K_1)\oplus Po(K_2).
     \end{equation*}

\end{coro}

\begin{proof}(See also \cite{jlc} Theorem 3.7).
    The map $\varepsilon_{K_i}^L:(Po(K_i))\to Cl(L)$ is injective for $i=1,2$. Indeed, if $I\in \ker(\varepsilon_{K_1}^L)$, then $N_{L/K_1}(I\mathcal{O}_L)=I^{[L:K_1]}=I^{[K_2:\Q]}$ is principal. By Lemma \ref{l1.5} and the fact that the degrees are coprime, $I$ is already principal. So we have $\varepsilon_{K_i}^L(Po(K_i))\simeq Po(K_i).$   
     We already have the sum decomposition according to the previous Corollary. It remains to show that the intersection $\varepsilon_{K_1}^L(Po(K_1))\cap\varepsilon_{K_2}^L(Po(K_2))$ is trivial in $Po(L)$.

By Lemma \ref{l1.5} if $p$ divides $|Po(K_i)|$ then it divides $[K_i:K]$. Consequently, as the degrees $[K_1:\Q]$ and $[K_2:\Q]$ are coprime, we conclude that the sum is direct:
    \begin{equation*}
        Po(L)\simeq Po(K_1)\oplus Po(K_2).
    \end{equation*}
\end{proof}

\section{On P\'olya fields}

\begin{prop}\label{2.1}
    Let $l$ a prime number and $n\in\N^*$. There exist infinitely many cyclic extensions $K/\Q$ of degree $l^n$ sutch that $Po(K)=\{0\}.$
\end{prop}
\begin{proof}
    Let $p$ a prime number verifying $p\equiv 1 \mod l^n$. Let $K$ the subfield of $\Q(\zeta_p)$ of degree $l^n$ over $\Q$. By \cite{wash}, Theorem 10.4, $|Cl(K)|\not\equiv 0\mod l$, and so, by Lemma \ref{l1.5}, $Po(K)=\{0\}$. We conclude with Dirichelt’s Theorem on arithmetic progressions.
\end{proof}

\begin{coro}
    Let $G$ an abelian group. There exist infinitely many P\'olya fields, abelian over $\Q$, tamely ramified, such that $Gal(K/\Q)\simeq G.$
\end{coro}
\begin{proof}
As $G$ is abelian, it is the direct product of its $p$-Sylow subgroups: $G=H_1\times\dots \times H_r$. We start to show that each $p_i$-Sylow subgroup $H_i$ of $G$ is the Galois group of a certain P\'olya field. Let $H_i=\Z/p_i^{e_1}\Z\times\dots\times \Z/p_i^{e_n}\Z$, by Proposition \ref{2.1}, for all $1\le j\le n$, there exist infinitely many cyclic P\'olya fields of degree $p_i^{e_j}$. Consequently, by Corollary \ref{1.4}, there exist infinitely many P\'olya fields, tamely ramified with Galois group $H_i$. Let $K_i$ such a field for $1\le i\le r$. When $i\ne j$, $[K_i:\Q]$ and $[K_j:\Q]$ are coprime and so, by Corollary \ref{1.4} applied to $K_1,\dots,K_r$, we have
\begin{equation*}
    Po(L)=\varepsilon_{K_1}^L(Po(K_1))\cdot\dots \cdot\varepsilon_{K_r}^L (Po(K_r))=\{1\}
\end{equation*}
where $L$ is the compositum $L=K_1\dots K_r$.

\end{proof}

\section{Yagahi's Theorem}
\begin{thm}[\cite{ya}]
    Let $l$ a prime number and $G$ an abelian group of exponent $l^m$, $m\ge 0$. Then for all $n\ge m$, $n\ge 1$, there exist infinitely many cyclic extensions of $\Q$ of degree $l^n$ whose $l$-class groups are isomorphic to $G$.
\end{thm}
In fact, in his paper, Yahagi shows refined result, especially concerning P\'olya groups. We give here a sketch of the proof for the sake of keeping this paper self-contained.
\begin{thm}
    Let $l$ a prime number and $G$ an abelian group of exponent $l^m$, $m\ge 0$. Then for all $n\ge m$, $n\ge 1$, there exist infinitely many cyclic extensions of $\Q$ of degree $l^n$ whose P\'olya groups are isomorphic to $G$.
\end{thm}
\begin{proof}
    For a cyclic extension $F/k$ of a number field $k$, and $\tau$ a generator of its Galois group, Yahagi defines the set ${S_F^{(\tau)}}_s=\{c\in S_F\;$; $c$ contains an ideal $\mathfrak{a}$ of $F$ such that $\mathfrak{a}^\tau=\mathfrak{a}\}$ (\cite{ya} p.276) where $S_F$ is the $l$-class group of $F$. When $k=\Q$ this set corresponds with the $l$-part of the P\'olya group of $F$, for it is generated by classes of ambiguous ideals. As we work with $l$-extensions, $Po(F)\otimes\Z_l=Po(F).$
    
    By class field theory, if $S_F={S_F^{(\tau)}}_s$, then ${S_F^{(\tau)}}_s\simeq G(K/F)$ where $K$ is the maximal abelian $l$-extension of $\Q$ contained in the genus field of $F/\Q$.
    
    Yahagi starts to build an abelian extension $K/\Q$, and a subfield $F$, cyclic over $\Q$ such that $K$ is exactly the maximal abelian $l$-extension of $\Q$ contained in the genus field of $F/\Q$ and $G(K/F)\simeq G$ (\cite{ya} p. 280-281).

    Then, considering the map $\varphi:S_{F}^{(\tau)}\to S_{F}/S_{F}^{1-\tau}$ so that the diagram below is commutative:
    
\begin{equation*}
      \xymatrix{
    S_F^{(\tau)} \ar[r] \ar[rd]_{\varphi} & S_F \ar[d] \\
   & S_F/S_F^{1-\tau}
  }
  \end{equation*}
  Yahagi shows that the image of ${S_F^{(\tau)}}_s$ by $\varphi$ is isomorphic to $S_{F}/S_F^{1-\tau}$, which implies that $\varphi$ is bijective and so $S_F={S_F^{(\tau)}}_s$.
\end{proof}
We can now deduce our main result from Yagahi's Theorem.
\begin{prop}
Let $G$ a finite abelian group. Then there exist infinitely many abelian extensions $E/\Q$ so that $Po(E)\simeq G.$
\end{prop}
\begin{remrk}[\cite{al}, Ex.1] 
    For the case $G=\{0\}$, it suffices to consider cyclotomic fields.
\end{remrk}
\begin{proof}
    Since $G$ is abelian, it is isomorphic to a direct product of its $p$-Sylow subgroups: 
    $G\simeq H_1\times\dots\times H_r.$ For all $1\le i \le r$,  Yagahi's Theorem gives us infinitely many abelian extensions of $\Q$ whose P\'olya groups are isomorphic to $H_i$.
    For $1\le i\le r$, let $K_i$ such an extension, with P\'olya group isomorphic to $H_i$. Then for all $i\ne j$, $([K_i:\Q],[K_j:\Q])=1$ and so, by Corollary \ref{1.7},
    \begin{equation*}
        Po(K_1\dots K_r)\simeq Po(K_1)\times\dots\times Po(K_r)\simeq G.
    \end{equation*}
\end{proof}

This study leads us to define the following mathematical object: for a number field $K$, Galois over $\Q$, we set $\Tilde{Cl(K)}:=Cl(K)/Po(K)$. 
Then, for such a $K$, it is interesting to find whether this relative class group is trivial or not. Here
comes a natural conjecture:
\begin{conj}
    There exist infinitely many number fields $K$, Galois over $\Q$, such that $\Tilde{Cl(K)}=\{0\}.$
\end{conj}
\begin{remrk}
    Chabert proved in \cite{jlc2}  that there exist infinitely many non Galois number fields $K$ such that $Cl(K)=Po(K)$.
\end{remrk}
However, in the case of imaginary quadratic fields, the problem seems easier to solve. Indeed just as we can prove that there exist a finite number of imaginary quadratic fields with class number one, so we show that there exist a finite number of imaginary quadratic fields with trivial relative class group.
\begin{thm}
    There exist a finite number of imaginary quadratic fields $K/\Q$ such that $\Tilde{Cl(K)}=\{0\}$.
\end{thm}
\begin{proof}
    Let $K$ an imaginary quadratic field, $h_K$ its class number and $d_K$ its discriminant. First, by
    Brauer-Siegel's Theorem \cite{Lang}, for all $\varepsilon>0$ there exists a positive constant $C_\varepsilon$ sucht that $h_K>C_\varepsilon d_K^{1/2-\varepsilon}$.
   Moreover, by Dirichlet's Theorem (\cite{wash} p.38) about class number for imaginary quadratic fields:
    \begin{equation*}
        \frac{2\pi h_K}{w\sqrt{|d_K|}}=L(1,\chi),
    \end{equation*}
    where $w$ stands for the number of roots of unity in $K$. From this equality we conclude that $h_K<<\sqrt{|d_K|}\log(|d_K|)$.
    Considering these two results, we show that 
    \begin{equation*}
        \frac{\log h_K}{\log |d_K|}\underset{|d_K|\to+\infty}{\to}\frac{1}{2}.
    \end{equation*}
    So, in term of growth, $h_K\sim\sqrt{|d_K|}.$\\
    
    For the P\'olya group of an imaginary quadratic field $K$, Hilbert \cite{Hil} proved that 
    \begin{equation*}
        |Po(K)|=2^{s_K-1},
    \end{equation*}
    where $s_K$ denotes the number of ramified prime numbers in the extension $K/\Q$. If $K=\Q(\sqrt{d})$, with $d$ squarefree, then the discriminant $d_K$ is $d$ or $4d$ depending on $d\equiv 1\mod 4$ or $d\equiv 2,3\mod 4$. If $2$ is ramified, $|d_K|=2^m p_1\dots p_{s_K-1}$, where the $p_i$'s are distinct odd primes, and $m\le 3.$ So, $|d_K|\ge2^3P_{s_K-1}$ where $P_{s_K-1}$ is the product of the first $s_K-1$ odd primes. On the one hand, when $s_K$ is not bounded, $\sqrt{P_{s_K-1}}\ge 2^{s_k-1}$ and so $\underset{d_K\to+\infty}{\lim}\frac{|Po(K)|}{\sqrt{|d_K|}}=0$.
    On the other hand, when $s_K$ is bounded, having $|d_K|\to +\infty$ means that at least one $p_i$ 
    is not bounded, and so $\underset{|d_K|\to+\infty}{\lim}\frac{|Po(K)|}{\sqrt{|d_K|}}=0$.  Consequently, the P\'olya group and the class group will only correspond for bounded discriminants with bounded $s_K$. 
    By Hermite's Theorem \cite{Her}, there are only finitely many number fields of bounded discriminant $d_K$, hence there is a finite number of imaginary quadratic fields with discriminant $d_K$.
    
\end{proof}

We can generalize  this result to the case of abelian CM fields.
\begin{thm}\label{3.8}
    Let $K$ runs through an infinite sequence of abelian CM fields. Let $m_K$, $h_K$ the conductor and the class number of $K$ and let $h_{K}^+$ the class number of $K^+$. Then,
    \begin{equation*}
        \underset{m_K\to+\infty}{\lim}\frac{\log(h_K^-)}{\log(\sqrt{|d_K|})}=1,
    \end{equation*}
    where $h_K^-=h_K/h_K^+$.
\end{thm}

In order to prove this Theorem we need to discuss about the discriminant $d_K$ of an abelian CM field $K$. For every abelian number field $K$, we will denote $m_K$ the conductor of $K$.
\begin{thm}[\cite{vj} 4.1]\label{3.9} Let $K$ an abelian number field of conductor $m_K=\prod_{i=1}^lp_i^{\alpha_i}$, where the $p_i$'s are prime number and $\alpha_i$'s and $l$ are positive integers. Let $d_K$ denotes the discriminant of $K$. Then,
\begin{equation*}
    |d_K|=(\prod_{i=1}^lp_i^{\alpha_i-\lambda_i})^{[K:\Q]}
\end{equation*}
where 
\begin{equation*}
    \lambda_i=\frac{p_i^{\alpha_i-(p_i,2)}-1+(p_i-1)/u_i}{p_i^{\alpha_i-(p_i,2)}(p_i-1)}
\end{equation*}
and
\begin{equation*}
    u_i=\frac{[K\cdot\Q(\zeta_{m_K/p_i^{\alpha_i}}):\Q(\zeta_{m_K/p_i^{\alpha_i}})]}{p_i^{\alpha_i-1}}.
\end{equation*}

\end{thm}
\begin{coro}\label{3.10}
    With the notations above, $\lambda_i$ is bounded by 2 for $i=1,\dots,l$.
\end{coro}
\begin{proof}
    We have:
    \begin{equation*}
        \lambda_i=\frac{1}{(p_i-1)}-\frac{1}{p_i^{\alpha_i-(p_i,2)}(p_i-1)}+\frac{p_i^{\alpha_i-1}}{p_i^{\alpha_i-(p_i,2)}[K\cdot\Q(\zeta_{m_K/p_i^{\alpha_i}}):\Q(\zeta_{m_K/p_i^{\alpha_i}})]}.
    \end{equation*}
    If $p_i$ is odd, then $\lambda_i\le\frac{1}{(p_i-1)}+\frac{1}{[K\cdot\Q(\zeta_{m_K/p_i^{\alpha_i}}):\Q(\zeta_{m_K/p_i^{\alpha_i}})]} \le 2$.
    \\
    If $p_i=2$,  then $\lambda_i= 1-\frac{1}{2^{\alpha_i-2}} +\frac{2}{[K\cdot\Q(\zeta_{m_K/2^{\alpha_i}}):\Q(\zeta_{m_K/2^{\alpha_i}})]}$. As $\alpha_i\ge 1$, $K\not \subset \Q(\zeta_{m_K/2^{\alpha_i}})$ and $[K\cdot\Q(\zeta_{m_K/p_i^{\alpha_i}}):\Q(\zeta_{m_K/p_i^{\alpha_i}})]\ge 2$. So 
    \begin{equation*}
        \lambda_i\le 1-\frac{1}{2^{\alpha_i-2}}+1\le 2.
        \end{equation*}
\end{proof}
 \begin{lem}\label{3.11}
    With the notations of Theorem \ref{3.9}, suppose $K$ runs through an infinite sequence of abelian fields. Then:
    \begin{equation*}
        \frac{[K:\Q]}{\log|d_K|}\underset{m_K\to+\infty}{\to}0.
    \end{equation*}
 \end{lem}
 \begin{proof}Let $l\in \N^*$ and $p_1,\dots,p_l$, $l$ distinct primes. Let $m_K=\prod_{i=1}^lp_i^{\alpha_i}$ with $\alpha_i\ge1$ for all $ 1\le i\le l$. If the $p_i$'s are bounded, then at least one $\alpha_i-\lambda_i$ tends to $+\infty$ when $m_K$ tends to $+\infty$. 
     By Corollary \ref{3.10} the $\lambda_i$'s are bounded for $i=1,\dots,l$ so at least one $\alpha_i$ tends to $+\infty$ when $m_K$ tends to $+\infty$. With the discriminant formula of Theorem \ref{3.9} we have 
     \begin{equation*}
         \frac{[K:\Q]}{\log|d_K|}=\frac{[K:\Q]}{[K:\Q]\log\prod_{i=1}^lp_i^{\alpha_i-\lambda_i}}=\frac{1}{\log\prod_{i=1}^lp_i^{\alpha_i-\lambda_i}}\underset{m_K\to+\infty}{\to} 0.
     \end{equation*}
     If the $p_i's$ are not bounded, then $\log\prod_{i=1}^lp_i^{\alpha_i-\lambda_i}$ tends to $+\infty$ when $m_K$ tends to $+\infty$. We conclude with the same argument.
 \end{proof}

\begin{lem}\label{3.12}
    With the notations of Theorem \ref{3.9}, let $K$ runs through an infinite sequence of abelian CM fields. Then:
    \begin{equation*}
        \underset{m_K\to+\infty}{\lim}\frac{|d_{K^+}|}{|d_K|}=0.
    \end{equation*}
\end{lem}
\begin{proof}

    According to Theorem \ref{3.9} applied for $K$ and $K^+$  we have:
    $|d_K|=(\prod_{i=1}^lp_i^{\alpha_i-\lambda_i})^{[K:\Q]}$
    and
        $|d_{K^+}|=(\prod_{i=1}^lp_i^{\alpha_i-\lambda_i'})^{[K^+:\Q]}$. 
    Then,
    \begin{equation*}
         \frac{|d_K|}{|d_{K^+}|}=\prod_{i=1}^lp_i^{[K^+:\Q](\alpha_i+\lambda_i'-2\lambda_i)}\ge \prod_{i=1}^lp_i^{[K^+:\Q](\alpha_i-\lambda_i)},
    \end{equation*}
because $\lambda_i'\ge\lambda_i$ by definition.
   If the $p_i$'s are bounded, then at least one $\alpha_i$ tends to $+\infty$ when $m_K$ tends to $+\infty$ and $\frac{|d_K|}{|d_{K^+}|}\underset{m_K\to+\infty}{\to}+\infty$. If the $p_i$'s are not bounded, because $\alpha_i-\lambda_i\ge1$, we have \begin{equation*}
       \underset{m_K\to+\infty}{\lim}\prod_{i=1}^lp_i^{[K^+:\Q](\alpha_i-\lambda_i)}=+\infty,
   \end{equation*}
   and we can conclude.

 \end{proof}
 \begin{remrk}
     Let $m\in\N^*$, there exist a finite number of abelian number fields with conductor $m_K\le m$.
 \end{remrk}
\begin{proof} (Theorem \ref{3.8}) With the notations of Theorem \ref{3.9}, let us consider the family of abelian CM fields. By Lemma \ref{3.11} and the Brauer-Siegel Theorem \cite{Lang}, we have 
\begin{equation*}
    \frac{\log(h_KR_K)}{\log(\sqrt{|d_K|)}}\underset{m_K\to +\infty}{\to} 1,
\end{equation*}
where $R_K$ denotes the regulator of the field $K$ and $h_K$ denotes the class number of $K$. 

Similarly we have
\begin{equation*}
        \frac{\log(h_K^+R_K^+)}{\log(\sqrt{|d_K^+|)}}\underset{m_{K^+}\to +\infty}{\to} 1,
\end{equation*}
where $h_K^+$, $R_K^+$ and $d_K^+$ are the class number, the regulator and the discriminant of the maximal real subfield $K^+$ of $K$. Let $h_K^-=h_K/h_K^+$, then 
\begin{equation}
    \frac{\log(h_KR_K)}{\log(\sqrt{|d_K|)}}=\frac{\log(h_K^+)+\log(h_K^-)+\log(R_K)}{\log(\sqrt{|d_K|})}.
\end{equation}
For CM fields, the quotient of $R_K$ by $R_K^+$ is well known (see for instance \cite{wash} Proposition 4.16):
\begin{equation*}
    \frac{R_K}{R_K^+}=\frac{1}{Q}2^{\frac{1}{2}[K:\Q]-1},
\end{equation*} 
with $Q=1$ or 2,
    and so the equation (1) becomes:
    \begin{equation*}
         \frac{\log(h_KR_K)}{\log(\sqrt{|d_K|)}}=\frac{\log(h_K^+)+\log(h_K^-)+\log(R_K^+)+(\frac{1}{2}[K:\Q]-1)\log(2)-\log(Q)}{\log(\sqrt{|d_K|})}.
    \end{equation*}
    By Lemma \ref{3.11}
    \begin{equation*}
        \frac{(\frac{1}{2}[K:\Q]-1)\log(2)-\log(Q)}{\log(\sqrt{|d_K|})}\underset{m_K\to +\infty}{\to}0,
    \end{equation*}
    and by Lemma \ref{3.12},
    \begin{equation*}
        \frac{\log(h_K^+)+\log(R_K^+)}{\log(\sqrt{|d_K|})}\underset{m_K\to +\infty}{\to}0.
    \end{equation*}
Finally we have
\begin{equation*}
    \frac{\log(h_K^-)}{\log(\sqrt{|d_K|)}}\underset{m_K\to +\infty}{\to}1.
\end{equation*}
So, in term of growth, $h_K^-\sim \sqrt{|d_K|}$.\\
\end{proof}
\begin{thm}
      Let $K$ runs through an infinite sequence of abelian fields. Then:
   \begin{equation*}
       \frac{|Po(K)|}{\sqrt{|d_K|}}\underset{m_K\to+\infty}{\to}0.
   \end{equation*}
\end{thm}
\begin{proof}
We know that $|Po(K)|$ divides $[K:\Q]^{l}$,  where $l$ is such that $m_K=\prod_{i=1}^lp_i^{\alpha_i}$, and we want to show that 
\begin{equation*}
    \frac{[K:\Q]^{l}}{\sqrt{|d_K|}}\underset{m_K\to+\infty}{\to}0.
\end{equation*}

By Theorem \ref{3.9} $\sqrt{|d_K|}=(\prod_{i=1}^lp_i^{\frac{\alpha_i-\lambda_i}{2}})^{[K:\Q]}$, so if the $p_i$'s are bounded, then at least one $\alpha_i$ tends to $+\infty$ when $m_K$ tends to $+\infty$ and 
\begin{equation*}
    \underset{m_K\to+\infty}{\lim}\frac{[K:\Q]^{l}}{\sqrt{|d_K|}}=0.
\end{equation*}
If the $p_i$'s are not bounded, then $(\prod_{i=1}^lp_i^{\frac{\alpha_i-\lambda_i}{2}})^{[K:\Q]}>> [K:\Q]^l$ and 
\begin{equation*}
    \underset{m_K\to+\infty}{\lim}\frac{[K:\Q]^{l}}{\sqrt{|d_K|}}=0.
\end{equation*}
Finally,
\begin{equation*}
    \underset{m_K\to+\infty}{\lim}\frac{|Po(K)|}{\sqrt{|d_K|}}=0.
\end{equation*}

\end{proof}
\begin{coro}
   There exist a finite number of abelian CM fields $K/\Q$ such that $$\Tilde{Cl(K)}=\{0\}.$$
\end{coro}

Concerning real quadratic fields, the problem is different since we do not control the behavior of $R_K$. The regulator may have the same growth asymptotically as $\sqrt{|d_K|}$.  

\begin{conj}
    There exists an infinite sequences of real quadratic fields such that 
    \begin{equation*}
        \underset{m_K\to+\infty}{\lim}\frac{\log(R_K)}{\log(\sqrt{|d_K|})}=1.
    \end{equation*}
\end{conj}
In fact, we can construct two infinite sequence of real abelian quadratic fields such that $$\frac{\log(R_K)}{\log(\sqrt{|d_K|})}\underset{m_K\to+\infty}{\to}0.$$ For this we will give some preliminary results.
\begin{thm}[\cite{Esterman} or \cite{HB}]\label{3.17}
    There exist infinitely many integers $n\in\N$ such that $n^2+1$ is square-free.
\end{thm}
\begin{lem}\label{3.18}
    There exist infinitely many integers $n$ such that $\varepsilon:=n+\sqrt{n^2+1}$ is the fundamental unit of the real quadratic field $\Q(\sqrt{n^2+1})$, where $n^2+1$ is square-free.
\end{lem}
\begin{proof}
    Theorem \ref{3.17} gives us infinitely many $n$ such that $n^2+1$ is square-free. Let $n$ such an integer. 
    If $n^2+1\equiv 2\mod 4$, then $\Z[\varepsilon]$ is the ring of integers of $\Q(\sqrt{n^2+1})$ and  Louboutin shows that the fundamental unit of $\Z[\varepsilon]$ is $\varepsilon$
    (\cite{Louboutin}, Thm 1). \\
    If $n^2+1\equiv1\mod 4$, 
    let $\eta$ be the fundamental unit of $\Q(\sqrt{n^2+1})$. Let $d_\eta$ (resp. $d_\varepsilon$) the discriminant of the order $\Z[\eta]$ (resp. $\Z[\varepsilon])$. 
    By definition, $d_\varepsilon=(\varepsilon-\varepsilon')^2$ where $\varepsilon'$ is the conjugate of $\varepsilon$. Here $\varepsilon'=n-\sqrt{n^2+1}$ and $d_\varepsilon=(\varepsilon+\varepsilon^{-1})^2=4(n^2+1)$.\\
    If $\varepsilon=\eta^m$ with $m\ge2$, since $\eta>1$:    \begin{equation*}
        (\eta^2-\eta^{-2})^2\le (\eta^m-\eta^{-m})^2\le (\varepsilon+\varepsilon^{-1})^2=d_\varepsilon\le4d_\eta\le 4(\eta+\eta^{-1})^2.
    \end{equation*}
    Then, $0\le (\eta-\eta^{-1})^2\le 4$ i.e. $\eta-\eta^{-1}\le 2$, and so $\eta\le 1+\sqrt{2}$.\\
   Hence we have:
    \begin{equation*}
        \eta^m+\eta^{-m}=\varepsilon+\varepsilon^{-1}\le2(\eta+\eta^{-1}),
    \end{equation*}
    and 
    \begin{equation*}
        \varepsilon=n+\sqrt{n^2+1}<2\eta + 2.
    \end{equation*}
    So \begin{equation*}
        \eta>\frac{n+\sqrt{n^2+1}}{2}-1 \;\text{and}\;\eta\le1+\sqrt{2}.
    \end{equation*}
    When $n$ is big enough, we obtain a contradiction and finally $\varepsilon=\eta$.
    \end{proof}
\begin{prop}
        There exist infinitely many integers  $n\in\N$ such that $4n^2-1$ is square-free.
    \end{prop}
    \begin{proof}
       Let $N\in \N$ and $p$ a prime number. Define 
$S_N=\{n\le N\;| \;4n^2-1 \;\text{is square-free}\}$ and $S_{N,p}=\{n\le N\;| \; 4n^2\equiv1\mod p^2\}$. Then,
\begin{equation*}
    |S_N|\ge N-\sum_{3\le p\le N}S_{N,p}.
\end{equation*}
Moreover, 
\begin{align*}
    |S_{N,p}|&=\underset{4n^2\equiv 1\mod p^2}{\sum_{n\le N}}1\\
&\le(1+\left[\frac{N}{p^2}\right])\underset{4x^2\equiv1\mod p^2}{\sum_{x\mod p^2}}1\le 2+\frac{2N}{p^2}.
\end{align*}
Consequently,
\begin{equation*}
    |S_N|\ge N-2\sum_{p\le N}1 - 2N\sum_{3\le p\le N}\frac{1}{p^2}.
\end{equation*}
By prime number Theorem we obtain,
\begin{equation*}
    |S_N|\le N(1-\sum_{p\ge3}\frac{1}{p^2}) + o(N),
\end{equation*}
and \begin{equation*}
    \sum_{p\ge 3}\frac{1}{p^2}\le \sum_{n=3}^\infty \frac{1}{n^2}=\frac{\pi^2}{6}-1-\frac{1}{4}\le\frac{5}{12}.
\end{equation*}
Finally,
\begin{equation*}
    |S_N|\ge \frac{N}{6}+o(N)
\end{equation*}
which prove the proposition.
 \end{proof}
 \begin{lem}\label{3.20}
      There exist infinitely many integers $n$ such that $\varepsilon:=2n+\sqrt{4n^2-1}$ is the fundamental unit of the real quadratic field $\Q(\sqrt{4n^2-1})$, where $4n^2-1$ is square-free.
 \end{lem}
 \begin{proof}
     It is a direct application of Louboutin's Theorem (\cite{Louboutin}, Thm 1), because the ring of integers of $\Q(\sqrt{4n^2-1})$ is precisely $\Z[2n+\sqrt{4n^2-1}]=\Z[\sqrt{4n^2-1}]$.
 \end{proof}
    \begin{prop}
        There exist an infinite sequence of real quadratic fields $K/\Q$ such that 
        \begin{equation*}
        \frac{\log(R_K)}{\log(\sqrt{|d_K|})}\underset{m_K\to+\infty}{\to}0.    
        \end{equation*}
    \end{prop}
    \begin{proof} 
    We present two such infinite sequences.
          By Lemma \ref{3.18} (respectively Lemma \ref{3.20} ) and the fact that $R_K$ is the logarithm of the fundamental unit, there exist infinitely many integers $n$ such that $\log(R_K)\sim\log(\log(n))$ and $\log(\sqrt{|d_K|})\sim \log(n)$, where $K=\Q(\sqrt{n^2+1})$ (respectively $K=\Q(\sqrt{4n^2-1})$). Because $m_K$ grows with $n$, the result holds.  
    \end{proof}
      \begin{coro}
          There exist infinitely many real quadratic fields $K/\Q$, such that 
          \begin{equation*}
              \frac{|\tilde{Cl(K)}|}{\log(\sqrt{|d_K|})}\underset{m_K\to+\infty}{\to}1
          \end{equation*}
      \end{coro}
  
\newpage
\printbibliography
\end{document}